\documentclass[a4paper,10pt]{article}
\usepackage{amsfonts,amsmath,amssymb,amsthm,verbatim,placeins,comment,placeins,caption}
\usepackage[dvips]{graphicx}
\usepackage{natbib}
\usepackage{bbm}
\usepackage{color}
\usepackage[usenames,dvipsnames,svgnames,table]{xcolor}
\usepackage{geometry}
\usepackage[cp1250]{inputenc}
\usepackage{lmodern}
\usepackage[OT2,T1]{fontenc}
\usepackage{graphicx}
\usepackage{caption,subfig}
\usepackage{fancyhdr}
\usepackage[style=english]{csquotes}

\newtheorem{theorem}{Theorem}[section]
\newtheorem{lemma}[theorem]{Lemma}

\theoremstyle{definition}

\theoremstyle{remark}
\newtheorem{remark}[theorem]{Remark}

\numberwithin{equation}{section}

\renewcommand{\Re}{\operatorname{Re}}

\newcommand{\R}{\mathbb{R}}

\newcommand{\Corr}{\mathop{\mathrm{Corr}}}
\newcommand{\I}{\mathop{\mathrm{I}}}

\newcommand{\rg}{\mathop{\mathrm{\mathfrak{rg}}}}

\newcommand{\fs}{\mathop{\mathrm{\mathfrak{fs}}}}

\newcommand{\p}{\mathop{\mathrm{\mathfrak{p}}}}

\newcommand{\m}{\mathop{\mathrm{\mathfrak{m}}}}
\newcommand{\s}{\mathop{\mathrm{\mathfrak{s}}}}
\newcommand{\G}{\mathop{\mathcal{G}}}
\newcommand{\Exp}{\mathbb E}

\begin{document}

\title{\textbf{Heavy-tailed fractional Pearson diffusions}}
\date{}
\maketitle

\centerline{\author{\textbf{N.N. Leonenko}$^{\textrm{a},}$, \textbf{I. Papi\'c}$^{\textrm{b}}$, \textbf{A. Sikorskii}$^{\textrm{c},}$\footnote{Corresponding author.}, \textbf{N. \v{S}uvak}$^{\textrm{b}}$}}

{\footnotesize{
$$\begin{tabular}{l}
  $^{\textrm{a}}$ \emph{School of Mathematics, Cardiff University, Senghennydd Road, Cardiff CF244AG, UK} \\
  $^{\textrm{b}}$ \emph{Department of Mathematics, J.J. Strossmayer University of Osijek, Trg Ljudevita Gaja 6, HR-31000 Osijek, Croatia} \\
  $^{\textrm{c}}$ \emph{619 Red Cedar Road, Department of Statistics and Probability, Michigan State University, East Lansing, MI 48824, USA} \\
\end{tabular}$$}}

\bigskip

\noindent\textbf{Abstract.}
We define heavy-tailed fractional reciprocal gamma and Fisher-Snedecor diffusions by a non-Markovian time change in the corresponding Pearson diffusions. Pearson diffusions are governed by the backward Kolmogorov equations with space-varying polynomial coefficients and are widely used in applications. The corresponding fractional reciprocal gamma and Fisher-Snedecor diffusions are governed by the fractional backward Kolmogorov equations and have heavy-tailed marginal distributions in the steady state. We derive the explicit expressions for the transition densities of the fractional reciprocal gamma and Fisher-Snedecor diffusions and strong solutions of the associated Cauchy problems for the fractional backward Kolmogorov equation.

\bigskip


\bigskip


\bigskip

\section{Introduction} \label{intro}

Recent developments in fractional processes including fractional diffusions have been motivated by applications that require modeling of phenomena in heterogeneous media using fractional partial differential equations, see \cite{Kochubey1989, Meerschaert2004, MainardiWaves}. Stochastic processes governed by these equations are found in many applications in science, engineering and finance (see \cite{Gorenflo2003, Magdziarz2009, MetzlerKlafter2000, MetzlerKlafter2004, Scalas2006, Stanislavsky2009}). In hydrology, the fractional time derivative models sticking and trapping of contaminant particles in a porous medium (\cite{Meerschaertetall2003}) or a river flow (\cite{Chakrabortyetall2009}). In finance, the fractional derivative in time models delays between trades (\cite{Scalas2006}), and has been used to develop the Black-Scholes formalism in this context (\cite{Magdziarz2009, Stanislavsky2009}). In statistical physics, fractional time derivative appears in the equation for a continuous time random walk limit and reflects random waiting times between particle jumps (\cite{Meerschaert2004, MetzlerKlafter2000, MetzlerKlafter2004}). For recent developments in treating the space-time fractional diffusion equations we refer to \cite{Nane2016103} and \cite{MagdziarzProceedings}. Detailed discussion of such equations is also found in \cite{Chen2012479}.  An application of fractional partial differential equations in the framework of spherical random fields is studied in \cite{D'Ovidio2016146}. For new results on semi-Markov dynamics and Kolmogorov's integro-differential equations we refer to \cite{MeerschaertArXiv} and \cite{OrsingherArXiv}.

Pearson diffusions is a special class of diffusion processes governed by the backward Kolmogorov equation
\begin{equation}\label{backward}
        \frac{\partial p(x,t;y)}{\partial t}= \mu(y)\frac{\partial p(x,t;y)}{\partial y}+\frac{\sigma^2(y)}{2}\frac{\partial^2 p(x,t;y)}{\partial y^2}.
        \end{equation}
with varying polynomial coefficients: $\mu (x)$ is polynomial of the first degree, and $\sigma ^2(x)$ is polynomial of at most second degree.  Stationary distributions of these diffusions  belong to the class of Pearson distributions (see \cite{PearsonTables_1914}). The study of classical Markovian Pearson diffusions began with Kolmogorov (in \cite{Shiryayev1992}) and \cite{Wong1964}, and continued in \cite{FormanSorensen_PD, LeonenkoSuvak_S_2010} and \cite{AvramLeonenkoSuvak_FS_PE_2011, AvramLeonenkoSuvak_FS_TSH_2012, AvramLeonenkoSuvak_WPD_2013, AvramLeonenkoSuvak_FS_SR_2013}. Diffusion processes from this family include the famous Ornstein-Uhlenbeck process (\cite{OU}) and the Cox-Ingersoll-Ross (CIR) process (\cite{CIR}) that are widely used in applications.  Other Pearson diffusions are Jacobi diffusion with beta stationary distribution, reciprocal gamma, Fisher-Snedecor and Student diffusions, named after their stationary distributions. The last three have heavy-tailed stationary distributions and therefore are known as heavy-tailed Pearson diffusions.

Pearson diffusions can be defined as stochastic processes by specifying their Markovian nature, their transition density via equation \eqref{backward}, and the distribution of the initial value. In contrast, fractional diffusions, due to the non-Markovian nature, generally cannot be defined by the governing equations and the initial distributions alone. Alternative approaches are required to define fractional diffusions as the stochastic processes.
For the three non-heavy-tailed Pearson diffusions (OU, CIR, Jacobi), their fractional analogs were defined in \cite{LeonenkoMeerschaertSikorskii_FPD_2013}, where it was also shown that these processes are governed by the time-fractional Fokker-Planck and backward Kolmogorov equations with space-varying polynomial coefficients. The spectral theory for the generators of the non-heavy-tailed Pearson diffusions was used to obtain the explicit strong solutions to the corresponding fractional Cauchy problems. Cauchy problems for time-fractional equations with fractional derivative of distributed order are considered in \cite{Mijena}.

This paper extends the results from \cite{LeonenkoMeerschaertSikorskii_FPD_2013} to two of the heavy-tailed Pearson diffusions. Namely, we define fractional reciprocal gamma and Fisher-Snedecor diffusions by time changing the corresponding non-fractional heavy-tailed diffusions by the inverse of the standard stable subordinator. We also discuss an alternative definition using stochastic calculus for time-changed semimartingales developed in \cite{Kobayashi2011}. Next, we obtain the spectral representations for the transition densities of the fractional reciprocal gamma and Fisher-Snedecor diffusions  and describe the first- and second-order properties of these processes. Finally, we use the general result of Baeumer and Meerschaert (2001) that  links the solutions of time-fractional and non-fractional Cauchy problems under the appropriate conditions on the operator appearing in  these problems. In general, finding strong solutions of time-fractional differential equations is a difficult problem, especially if the coefficients are variable.  We are able to obtain the explicit strong solutions to the backward Kolmogorov equations using properties of the spectrum of the generator with polynomial coefficients. Since the structure of the spectrum of the generators of heavy-tailed Pearson diffusions is different and more complex than in the non-heavy-tailed case, different analytical methods are used to derive the corresponding spectral representations and strong solutions.

\section{Pearson diffusions} \label{PD}

Continuous distributions with densities satisfying the famous Pearson equation
\begin{equation}\label{Peq}
\frac{\p^{\prime}(x)}{\p(x)} = \frac{(a_{1} - 2b_{2})x +(a_{0} - b_{1})}{b_{2}x^{2} + b_{1}x + b_{0}}
\end{equation}
are called Pearson distributions (see \cite{PearsonTables_1914}). The family of Pearson distributions is categorized into six well known and widely applied parametric subfamilies: normal, gamma, beta, Fisher-Snedecor, reciprocal gamma and Student distributions. The last three distributions are heavy-tailed (see \cite{LeonenkoSuvak_RG_2010, LeonenkoSuvak_S_2010} and \cite{AvramLeonenkoSuvak_FS_SR_2013}) and therefore are of special interest in stochastic modeling. One set of models that produces these distributions uses classical Markovian diffusions driven by the Brownian motion $(W_{t}, \, t \geq 0)$. In this setting, the stochastic differential equation (SDE)
\begin{equation}\label{SDE}
dX_t = \mu(X_t)dt+ \sigma(X_t)dW_t, \,\, t \geq 0, \,\, \mu(x)=a_0+a_1 x, \,\, \sigma(x)=\sqrt{2b(x)}=\sqrt{2(b_{2}x^{2} + b_{1}x + b_{0})}, \quad t \geq 0,
\end{equation} has polynomial infinitesimal parameters that are related to the polynomials in the Pearson equation \eqref{Peq}: the drift $\mu(x)$ is  linear and the squared diffusion $\sigma^{2}(x)$ is at most quadratic. Because the coefficients are polynomial, classical results \cite[Theorem 5.2.1]{Oksendal} imply the existence and uniqueness of the strong solutions of SDE \eqref{SDE}, and the solutions  are called Pearson diffusions. Their stationary distributions belong to the Pearson family, and it is often convenient to re-parametrize
\begin{equation}\label{parametrize}
\mu(x)=a_{0}+a_{1}x=-\theta(x-\mu), \quad
\sigma^{2}(x) =  2\theta k (b_{2}x^{2} + b_{1}x + b_{0}),
\end{equation}
where $\mu \in \mathbb{R}$ is the stationary mean depending on coefficients of the Pearson equation \eqref{Peq},  $\theta > 0$ is the scaling of time determining the speed of the mean reversion,  and
 $k$ is a positive constant. Note that we need $\sigma^{2}(x)>0$ on the diffusion state space $(l, L)$.

Pearson diffusions could be categorized into six subfamilies, according to the degree of the polynomial $b(x)$ and, in the quadratic case $b(x)=b_{2}x^{2} + b_{1}x + b_{0}$, according to the sign of its leading coefficient $b_{2}$ and the sign of its discriminant $\Delta$:

\begin{itemize}
  \item constant $b(x)$ - Ornstein-Uhlenbeck (OU) process with normal stationary distribution,
  \item linear $b(x)$ - Cox-Ingersol-Ross (CIR) process with gamma stationary distribution,
  \item quadratic $b(x)$ with $b_{2}<0$ - Jacobi diffusion with beta stationary distribution,
  \item quadratic $b(x)$ with $b_{2}>0$ and $\Delta>0$ - Fisher-Snedecor (FS) diffusion with the Fisher-Snedecor stationary distribution,
  \item quadratic $b(x)$ with $b_{2}>0$ and $\Delta=0$ - reciprocal gamma (RG) diffusion with reciprocal gamma stationary distribution,
  \item quadratic $b(x)$ with $b_{2}>0$ and $\Delta<0$ - Student diffusion with the Student stationary distribution.
\end{itemize}

The first three types of Pearson diffusions have stationary distributions with all moments, and are studied in detail in \cite{LeonenkoMeerschaertSikorskii_FPD_2013}. In this paper we focus on analytical and probabilistic properties of the other two subfamilies, the FS and RG diffusions that have heavy-tailed stationary distributions. The sixth subfamily (Student diffusion) also has a heavy-tailed stationary distribution, but its spectral properties differ from those of the FS and RG diffusions, and Student fractional diffusion will be dealt with in a separate paper. Results concerning the spectral representation of the transition density of Pearson diffusions and estimation of their parameters could be found in the series of papers (\cite{AvramLeonenkoSuvak_FS_PE_2011, AvramLeonenkoSuvak_FS_TSH_2012, AvramLeonenkoSuvak_WPD_2013, AvramLeonenkoSuvak_FS_SR_2013,  LeonenkoSuvak_RG_2010, LeonenkoSuvak_S_2010}).

For a Markov process $\left(X_t,\,t\ge 0\right)$ with absolutely continuous transition probability distribution, let $p(x, t; y, s) = \frac{d}{dx} P(X_{t} \leq x \vert X_{s} = y)$ be the corresponding transition density. In this paper we consider only time-homogeneous diffusions for which $p(x, t; y, s) = p(x, t-s; y, 0)$ for $t > s$ and we write $p(x, t; y) = \frac{d}{dx}P(X_t \leq x | X_0=y)$. With $\mu(x)$ and $\sigma(x)$ given by \eqref{parametrize}, the transition densities satisfy the following partial differential equations (PDEs) with the point-source initial conditions.

\bigskip

\noindent{\bf{Forward Kolmogorov or Fokker-Planck equation:}}
        \begin{equation*}\label{KFE}
        \frac{\partial p(x,t;y)}{\partial t}= -\frac{\partial}{\partial x}\left(\mu(x) p(x,t;y) \right)+\frac{1}{2}\frac{\partial^2}{\partial x^2}\left(
        \sigma^2(x) p(x,t;y) \right).
        \end{equation*}
        In this PDE, the current state $y$ is constant, and the equation describes the "forward evolution" of the diffusion. The second-order differential operator in this equation is the  Fokker-Planck operator
        \begin{equation*}\label{FPoperator}
        \mathcal{L} g(x) = -\frac{\partial}{\partial x} \left( \mu(x) g(x) \right) + \frac{1}{2} \frac{\partial^{2}}{\partial x^{2}} \left( \sigma^2(x) g(x) \right).
        \end{equation*}

\bigskip

\noindent{\bf{Backward Kolmogorov equation:}}
        \begin{equation*}\label{KBE}
        \frac{\partial p(x,t;y)}{\partial t}= \mu(y)\frac{\partial p(x,t;y)}{\partial y}+\frac{\sigma^2(y)}{2}\frac{\partial^2 p(x,t;y)}{\partial y^2}.
        \end{equation*}
        In this PDE, the future state $x$ is constant, and the equation describes the "backward evolution" of the diffusion. The second-order differential operator in this equation is the infinitesimal generator of the diffusion
        \begin{equation}\label{InfGen}
        \mathcal{G} g(y) = \left ( \mu(y) \frac{\partial}{\partial y}+\frac{\sigma^2(y)}{2}\frac{\partial^2}{\partial y^2} \right )g(y).
        \end{equation}
        According to \cite{McKean1956},  generator $\mathcal{G}$ is closed, generally unbounded, negative semidefinite, self-adjoint operator densely defined on the space $L^{2}\left( (l,L), \m \right)$ of square integrable functions with the weight $\m $ equal to the diffusion speed density:
				\begin{equation*} \m(x) = \frac{2}{\sigma^{2}(x) \s(x)},
				\end{equation*}
				where
				\begin{equation*} \s(x) = \exp{\left\{-2 \int\limits^{x} \frac{\mu(y)}{\sigma^{2}(x)} \, dy \right\}} \end{equation*} is the diffusion scale density (see \cite{BorodinSalminen} and \cite{KarlinTaylor2_1981}).
				The domain of the generator is the following space of functions:
        \begin{equation}\label{domain}
        \{ f \in L^{2}\left((l,L), \mathfrak{m}\right) \cap C^2\left( (l,L) \right) \colon \mathcal{G}f \in L^{2}\left( (l,L), \mathfrak{m} \right) \,\,\text{and } f \,\,\text {satisfies boundary  conditions at } l\text{ and } L \}.
        \end{equation}
For the FS and RG diffusions with the state space $(0, \infty )$, the boundary conditions are (see \cite{McKean1956}):
$$\lim\limits_{y \to 0} \frac{g^{\prime}(y)}{\s(y)} = \lim\limits_{y \to \infty} \frac{g^{\prime}(y)}{\s(y)} = 0. $$

\section{Fractional Cauchy problems}\label{fC}

Consider the fractional Cauchy problem involving the generator \eqref{InfGen}:
\begin{equation} \label{FbackwardCauchy}
\frac{\partial^{\alpha} q(y,t)}{\partial t^{\alpha}}= \mathcal{G} q(y,t), \quad q(y,0)=g(y),
\end{equation}
where $\displaystyle\frac{\partial^{\alpha}}{\partial t^{\alpha}}$  is the Caputo fractional derivative of order $0<\alpha<1$. It is defined as
$$\frac{d^{\alpha}f(x)}{dx^{\alpha}} = \frac{1}{\Gamma(1-\alpha)} \int\limits_{0}^{\infty} \frac{d}{dx} f(x-y) y^{-\alpha} \, dy,$$
or equivalently for  absolutely continuous functions as $$\frac{d^{\alpha}f(x)}{dx^{\alpha}} = \frac{1}{\Gamma(1-\alpha)} \int\limits_{0}^{x} (x-y)^{-\alpha} f'(y) \, dy.$$
For more details on fractional derivatives see \cite[Chapter 2]{MeerschaertSikorskii_2011}.

The ordinary (non-fractional) Cauchy problem for the generator is
\begin{equation} \label{backwardCauchy}
\frac{\partial q(y,t)}{\partial t}= \mathcal{G} q(y,t), \quad q(y,0)=g(y).
\end{equation}
As described in \cite{LeonenkoMeerschaertSikorskii_FPD_2013}, separation of variables approach could be used to find heuristic solutions to \eqref{FbackwardCauchy} and
\eqref{backwardCauchy} in the form $q(y,t)=T(t)\varphi (y)$,  where functions $T$ and $\varphi$ may depend on $x$ and $\alpha$.
  Then for \eqref{FbackwardCauchy} we have
\[
\frac{1}{T(t)}\frac{d^{\alpha}T(t)}{dt^{\alpha}}=\frac{\mathcal{G}\varphi(y)}{\varphi(y)},
\]
assuming that $T$ and $\varphi$ do not vanish. The last equation obviously holds if and only if both sides are equal to a constant denoted $-\lambda$ (so that $\lambda >0$), leading to two equations:
\begin{equation} \label{sep_1}
\frac{d^{\alpha}T(t)}{dt^{\alpha}}=-\lambda T(t)
\end{equation}
and
\begin{equation}\label{sep_2}
\mathcal{G} \varphi = -\lambda \varphi.
\end{equation}
In the case of non-fractional Cauchy problem \eqref{backwardCauchy}, \eqref{sep_2} is the same, while \eqref{sep_1} is replaced by
\begin{equation} \label{nsep_1}
\frac{dT(t)}{dt}=-\lambda T(t).
\end{equation}
Equations \eqref{sep_1} and \eqref{nsep_1} have well-known strong solutions,  $\mathcal {E}_\alpha (-\lambda t)$ and $\exp \{-\lambda t\}$, respectively, where
\begin{equation*}
\mathcal {E}_{\alpha}(-\lambda t^{\alpha}) = \sum\limits_{j=0}^{\infty} \frac{(-\lambda t^{\alpha})^{j}}{\Gamma(1+\alpha j)}
\end{equation*}
is the Mittag-Leffler function (see, for example, \cite{Simon_Mittag}).

Regarding the space part, both fractional and non-fractional Cauchy problems lead to the Sturm-Liouville problem for the negative generator $(-\mathcal G)$. Therefore, spectral properties of the generator of the corresponding diffusions are linked to the problem of finding strong solutions of the corresponding Cauchy problems.

For non-heavy-tailed diffusions considered in \cite{LeonenkoMeerschaertSikorskii_FPD_2013}, the spectrum of $(-\mathcal G)$ is purely discrete, consisting of infinitely many simple eigenvalues $(\lambda_{n}, \, n \in \mathbb{N})$. Corresponding eigenfunctions are classical systems of orthogonal polynomials: Hermite polynomials for OU process, Laguerre polynomials for CIR process and Jacobi polynomials for Jacobi diffusion. According to the spectral classification of \cite{Linetsky_2007spectral}, these non-heavy-tailed Pearson diffusions belong to the spectral category I. The summary of their spectral properties is found in \cite{LeonenkoMeerschaertSikorskii_FPD_2013}.

Generators of heavy-tailed Pearson diffusions have more complicated structure of the spectrum. The spectrum of the operator $(-\mathcal{G})$ in case of RG and FS diffusions consists of two disjoint parts - the discrete part and the absolutely continuous part of multiplicity one, see \cite{LeonenkoSuvak_RG_2010} and \cite{AvramLeonenkoSuvak_FS_SR_2013}. These two diffusions belong to the spectral category II in the Linetsky's classification (see \cite{Linetsky_2007spectral}).

In case of the Student diffusion the spectrum of the generator $(-\mathcal{G})$ consists of the discrete part and the absolutely continuous part of multiplicity two, see \cite{LeonenkoSuvak_S_2010}. This diffusion belongs to the Linetsky's spectral category III.

As mentioned in the Introduction, governing fractional equations alone may not always define stochastic processes. Therefore, we define fractional counterparts of Pearson diffusions via a non-Markovian time-change in Pearson diffusions of spectral category II . Furthermore, we derive the explicit expressions for their transition densities and discuss the first- and second-order properties of these fractional processes. Finally, we revisit the fractional Cauchy problems and obtain strong solutions for the corresponding fractional backward Kolmogorov equations.

\section{Fractional Pearson diffusions of spectral category II and their properties} \label{PD_spectrum}

\subsection{Fractional Pearson diffusions of spectral category II} \label{FPD}

Let $X = \left( X(t), \, t \geq 0 \right)$ be the Pearson diffusion solving  \eqref{SDE}. Introduce $(D_t, \, t \geq 0)$, the standard stable subordinator with index $0 < \alpha < 1$, independent of the process $X$. $D_{t}$ is a homogeneous L\`evy process with the Laplace transform
\[
\Exp [e^{-sD_t}]=\exp\{-ts^{\alpha}\}
\]
for $0<\alpha <1$. Its inverse process
\[
E_t=\inf \{x>0:D_x>t \}.
\]
is non-Markovian, non-decreasing, and for every $t$ random variable $E_t$ has a density, which will be denoted by $f_t(\cdot )$. The Laplace transform of this density is (see e.g., \cite{PSW2005})
\begin{equation} \label{Mittag}
\Exp [e^{-sE_t}]=\int_{0}^{\infty}e^{-sx}f_t(x)dx=\mathcal {E}_{\alpha}(-st^{\alpha}).
\end{equation}
 The density of $E_t$ is related to the density of the standard stable subordinator $D$ as follows. With $g_\alpha $ denoting the density of $D_{1}$
\begin{equation*}\label{etdensity}
f_t(u)=\frac {t}{\alpha }u^{-1-1/\alpha }g_\alpha (tu^{-1/\alpha}),
\end{equation*}
see \cite[p. 111]{MeerschaertSikorskii_2011}.
Define the fractional Pearson diffusion $(X_{\alpha}(t), \, t \geq 0)$ via time-change of the Pearson diffusion
\begin{equation}\label{fpddef}
X_{\alpha}(t) = X(E_t), \,\, t \ge 0.
\end{equation}
The process $X_\alpha(t)$ is non-Markovian. We define its transition density $p_{\alpha}(x,t;y)$ as
\begin{equation}\label{trans}
P(X_{\alpha}(t) \in B | X_{\alpha}(0)=y)=\int_{B}p_{\alpha}(x,t;y)dx
\end{equation}
for any Borel subset $B$ of $(l, L)$.

\subsection{Fractional Pearson diffusions as solutions of  the SDEs}

The fractional Pearson diffusions defined via a time-change of the ordinary (non-fractional) Pearson diffusions satisfy the special types of SDEs considered in \cite{Kobayashi2011}. Let $\{\mathcal {F}_t, t\ge 0\}$ be the natural filtration associated with the Brownian motion from equation \eqref{SDE}. Since $E_t$ has almost surely continuous sample paths, for any $t>0$ $[E_{t-}, E_t]$ contains only one point, and the Brownian motion and the solution of SDE \eqref{SDE} $X(t)$  are in synchronization with the time change $E_t$ as defined in \cite{Kobayashi2011}. This synchronization is key to the stochastic calculus for the semimartingale $B(E_t)$ with respect to filtration $\{\mathcal {F}_{E_t}, t\ge 0\}$. See also \cite{Schilling} for the discussion on martingale properties of $B(E_t)$ and other processes obtained as time-changes of the Brownian motion using inverse  subordinators.

Specialized to fractional Pearson diffusions, according to the duality Theorem  \cite[ Theorem 4.2, part (1)]{Kobayashi2011},  when the process $X(t)$ satisfies SDE \eqref{SDE} with the initial condition $X(0)=X_0$, the time-changed process $X_\alpha (t)=X(E_t)$ satisfies the SDE
\begin{equation}\label{fSDE}
dX_\alpha (t)=\mu (X_\alpha (t))dE_t+\sigma (X_\alpha (t))dB_{E_t}
\end{equation}
with the initial condition $X_\alpha (0)=X_0$.

Further, when $\mu $ and $\sigma^2 $ are polynomials of the first and second degree, respectively, as specified in \eqref{SDE},  \cite[Lemma 4.1]{Kobayashi2011} ensures the existence and uniqueness of the strong solution of \eqref{fSDE}, giving another possible definition of the fractional Pearson diffusions as solutions of this SDE.
The solution can be represented in the
following integral form:
\begin{equation*}
X_\alpha (t)=X(E_{t}) = \int\limits_{0}^{E_{t}} \left( a_{0}+a_{1}X(s)
\right) ds + \int\limits_{0}^{E_{t}} \sqrt{2(
b_{0}+b_{1}X(s)+b_{2}(X(s))^2)} dB(s).
\label{FPD_Ito_rep}
\end{equation*}
The integrals in this representation are the Lebesgue
and the It\^{o} integrals under the continuous time-change $t \mapsto
E_{t}$, in light of the change-of-variable formula from
\cite[Theorem 3.1]{Kobayashi2011}. This representation could be useful for simulating paths of fractional Pearson diffusions. The discrete schemes for the underlying densities and their error bounds can be found in \cite{Kelbert20161145}. For similar approaches to obtaining solutions of such SDEs we refer to \cite{Scalas2014385}.

\subsection{Spectral representation of the transition density of the fractional reciprocal gamma diffusion} \label{SSS_RG}
We begin this subsection by listing the necessary facts about the non-fractional reciprocal gamma diffusion, then use these facts to derive a representation for the transition density of the fractional reciprocal gamma diffusion.

The reciprocal gamma diffusion satisfies the SDE
\begin{equation*}
dX_{t} = -\theta \left(X_{t} - \frac{\gamma}{\beta - 1}\right)dt + \sqrt{\frac{2 \theta}{\beta-1} X_{t}^{2}} \, dW_{t}, \quad t \geq 0,  \label{RGsde}
\end{equation*}
with $\theta > 0$ and invariant density
\begin{equation}
\rg(x) = \frac{\gamma^{\beta}}{\Gamma(\beta)} \, x^{-\beta - 1} e^{-\frac{\gamma}{x}} {\I}_{\langle 0,\infty \rangle }(x) \label{RGdensity}
\end{equation}
with parameters $\beta > 1$ and $\gamma > 0$, where the latter requirement ensures the existence of the stationary mean $\gamma/(\beta-1)$.

If $\beta > 2$ the variance of the invariant distribution exists, and the autocorrelation function of the RG diffusion is
\begin{equation}
\rho(t) = \mathop{\mathrm{Corr}}(X_{s + t}, X_{s}) = e^{-\theta t}, \quad t \geq 0, \quad s \geq 0. \label{RGacf}
\end{equation}

The generator is
\begin{equation*}
\G f(x) = \frac{\theta}{\beta - 1} \, x^{\beta + 1}e^{\frac{\gamma}{x}} \, \frac{d}{dx} \, \left(x^{-\beta + 1} e^{-\frac{\gamma}{x}} f^{\prime}(x) \right) =
\end{equation*}
\begin{equation*}
\hspace*{25mm} = \frac{\theta}{\beta-1} \, x^{2} f^{\prime \prime }(x) - \theta \left(x - \frac{\gamma}{\beta-1} \right) \, f^{\prime }(x), \quad x > 0. \label{RGD_InfGen}
\end{equation*}

In this case operator $(-\mathcal{G})$ has the finite discrete spectrum $\sigma_{d}(-\G) \subset [0, \Lambda )$, and the purely absolutely continuous spectrum $\sigma_{ac}(-\G)$ of multiplicity one in $( \Lambda, \infty )$, where $$\Lambda = \frac{\theta \beta^{2}}{4(\beta-1)}, \quad \beta > 1,$$ is the cutoff between the two parts of the spectrum.

The discrete part of the spectrum consists of the eigenvalues
\begin{equation}
\lambda_{n} = \frac{\theta}{\beta - 1} \, n (\beta - n), \quad n \in \left\{0, 1, \ldots, \left\lfloor \frac{\beta}{2} \right\rfloor \right\}, \quad \beta > 1. \label{RGD_diseigen}
\end{equation}
The corresponding eigenfunctions are Bessel polynomials given by the Rodrigues formula
\begin{equation*}\label{RGD_BesselRodForm}
\widetilde{B}_{n}(x) = x^{\beta+1} e^{\frac{\gamma}{x}} \frac{d^{n}}{dx^{n}} \, (x^{2n-(\beta + 1)} e^{-\frac{\gamma}{x}}), \quad n \in \left\{0, 1, \ldots, \left\lfloor \frac{\beta}{2} \right\rfloor \right\}, \quad \beta > 1.
\end{equation*}
The normalized Bessel polynomials are
\begin{equation} \label{RGD_NormBesselRodForm}
B_{n}(x) = K_{n} \, \widetilde{B}_{n}(x),
\end{equation}
where
\begin{equation*}
K_{n} = \frac{(-1)^{n}}{\gamma^{n}} \sqrt{\frac{(\beta - 2n)\Gamma(\beta)}{\Gamma(n+1) \Gamma(\beta - n + 1)}} = \frac{(-1)^{n}}{\gamma^{n}} \sqrt{\frac{\Gamma(\beta)}{n! \Gamma(\beta - 2n)} \, \left( \prod_{k=0}^{n-1} (\beta - n - k) \right)^{-1}}  \label{RGD_BesselNormConst}
\end{equation*}
is the normalizing constant.

Since the absolutely continuous part of the spectrum is of the form $( \Lambda, \infty )$, its elements could be parametrized as
\begin{equation*}
\lambda = \Lambda + \frac{\theta k^{2}}{\beta-1} = \frac{\theta}{\beta-1} \left( \frac{\beta^{2}}{4} + k^{2} \right), \quad \beta > 1, \quad k \in \mathbb{R}. \label{RGD_conteigen}
\end{equation*}
The spectral representation of the transition density of the RG diffusion consists of two parts:
\begin{equation}
p_{1}(x; x_{0}, t) = p_{d}(x; x_{0}, t) + p_{c}(x; x_{0}, t). \label{RGD_d+c}
\end{equation}
The discrete part of the spectral representation is
\begin{equation*}
p_{d}(x; x_{0}, t) = \rg(x) \, \sum\limits_{n=0}^{\left\lfloor \frac{\beta}{2} \right\rfloor}e^{-\lambda_{n} t} \, B_{n}(x_{0}) \, B_{n}(x),  \label{RGD_disspec}
\end{equation*}
where $\rg(\cdot)$ is the invariant density \eqref{RGdensity}, eigenvalues $\lambda_{n}$ are given by \eqref{RGD_diseigen} and normalized Bessel polynomials are given by \eqref{RGD_NormBesselRodForm}. The continuous part of the spectral representation is given in terms of the elements $\lambda$ of the absolutely continuous part of the spectrum of the operator $(-\mathcal{G})$
\begin{equation*}
p_{c}(x; x_{0}, t) = \rg(x) \, \frac{1}{4\pi} \int\limits_{\frac{\theta \beta^{2}}{4(\beta-1)}}^{\infty}e^{-\lambda t}\,b(\lambda)\,\psi(x, -\lambda)\,\psi(x_0, -\lambda)\,d\lambda,  \label{RGD_contspec}
\end{equation*}
where
\begin{equation}\label{RGparameters}
b(\lambda)=\frac{\gamma^{-\beta-1}}{k(\lambda)} \, \left| \frac{ \Gamma^{\frac{1}{2}}(\beta) \, \Gamma\left( -\frac{\beta}{2} + ik(\lambda) \right)}{\Gamma\left(2ik(\lambda)\right)} \right|^{2}, \qquad
k(\lambda) = -i \sqrt{\frac{\beta^{2}}{4} - \frac{\lambda (\beta - 1)}{\theta}}.
\end{equation}
The function
\begin{equation}\label{psi}
\psi(x, -\lambda) = \gamma^{\frac{\beta + 1}{2}} \,_{2}F_{0}\left( -\frac{\beta}{2} + ik(\lambda), -\frac{\beta}{2} - ik(\lambda); ; -\frac{x}{\gamma} \right)
\end{equation} is the solution of the Sturm-Liouville equation $\mathcal{G}f(x) = -\lambda f(x)$, $\lambda > 0$, where $_{2}F_{0}$ is the special case of the generalized hypergeometric function $_{p}F_{q}$ with $p=2$ and $q=0$, see \cite{Slater} or \cite{NIST}.

More details on the RG diffusion and this representation of the transition density are found in \cite{LeonenkoSuvak_RG_2010}. We now proceed with the statement of our result on the spectral representation of the transition density of the fractional RG diffusion.

\begin{theorem}
The transition density of the fractional RG diffusion is given by
\begin{equation}\label{frtranrg}
p_\alpha (x,t;x_0)=\sum \limits_{n=0}^{\lfloor\frac{\beta}{2}\rfloor} \mathfrak{rg}(x)\, B_n(x)\,B_n(x_0)\,
\mathcal{E}_{\alpha}(-\lambda_n t^{\alpha})
+\frac{\mathfrak{rg}(x)}{4\pi} \int\limits_{\frac{\theta \beta^2}{4(\beta-1)}}^{\infty}
\mathcal{E}_{\alpha}(-\lambda t^{\alpha})\,b(\lambda)\,\psi(x,-\lambda)\,\psi(x_0,-\lambda)\,d\lambda,
\end{equation}
where Bessel polynomials $B_n$ are given by equation \eqref{RGD_NormBesselRodForm}, the solution of the Sturm-Liouville equation $\psi$ is given by \eqref{psi} with $b(\lambda )$ given by \eqref{RGparameters}.
\end{theorem}

\begin{proof}
Since the Pearson diffusion $X(t)$ is independent of the time change $E_t$, using \eqref{RGD_d+c} and \eqref{Mittag} together with the Fubini argument, we have
\begin{align} \label{FRG_tonelli}
P(X_{\alpha}(t) \in B | X_{\alpha}(0)=x_0) &= \int_{0}^{\infty}P(X_{1}(\tau) \in B | X_{1}(0)=x_0)\,f_{t}(\tau)\,d\tau \nonumber \\
&=\int_{0}^{\infty}\int_{B}^{}p_1(x,\tau;x_0)\,f_t(\tau)\,dx \, d\tau  \nonumber \\
&=\int_{B}^{}\int_{0}^{\infty}(p_d(x,\tau;x_0)+p_c(x,\tau;x_0))\, f_t(\tau)\,d\tau \,dx
\end{align}
$$=\int_{B}^{} \left[\int\limits_{0}^{\infty} \sum \limits_{n=0}^{\lfloor\frac{\beta}{2}\rfloor} \mathfrak{rg}(x) \,B_n(x_0)\,B_n(x)\,e^{- \lambda_n \tau}\,f_t(\tau)\,d\tau
+\frac{\mathfrak{rg}(x)}{4\pi}\int_{0}^{\infty } \int\limits_{\frac{\theta \beta^2}{4(\beta-1)}}^{\infty}e^{-\lambda \tau}\,f_t(\tau)\,b(\lambda)\,\psi(x,-\lambda)\,\psi(x_0,-\lambda)\,d\lambda \, d\tau  \right ]dx
$$
\begin{equation} \label{FRG_fubini}
=\int_{B}^{}  \left [ \sum \limits_{n=0}^{\lfloor\frac{\beta}{2}\rfloor} \mathfrak{rg}(x) \,B_n(x)\,B_n(x_0)\,\mathcal{E}_{\alpha}(-\lambda_n t^{\alpha})
+\frac{\mathfrak{rg}(x)}{4\pi} \int\limits_{\frac{\theta \beta^2}{4(\beta-1)}}^{\infty} \mathcal{E}_{\alpha}(-\lambda t^{\alpha})\,b(\lambda)\,\psi(x,-\lambda)\,\psi(x_0,-\lambda)\,d\lambda  \right ]dx.
\end{equation}

The change of the order of integration in  \eqref{FRG_tonelli} is justified by the Fubini-Tonelli Theorem \cite[Theorem 8.8 (a)]{Rudin} since functions $p_1$ and $f_t$ are non-negative.
In contrast, change of the order of integration in \eqref{FRG_fubini} cannot be justified by the Fubini-Tonelli Theorem since the integrand
$$g(\lambda,\tau)=e^{-\lambda \tau}\,f_t(\tau)\,b(\lambda)\,\psi(x,-\lambda)\,\psi(x_0,-\lambda)$$
is not necessarily non-negative. To justify this step using the Fubini Theorem \cite[Theorem 8.8 (b-c)]{Rudin}, below we show that
\begin{equation}\label{RG_Fubini_proof}
\int\limits_{\frac{\theta \beta^2}{4(\beta-1)}}^{\infty}\int\limits_{0}^{\infty}|g(\lambda,\tau)| \,d \tau \, d\lambda < \infty.
\end{equation}
Let $$h(\lambda)=\mathcal{E}_{\alpha }(-\lambda t^{\alpha})\,b(\lambda)\,\psi(x,-\lambda)\,\psi(x_0,-\lambda).$$
Since

\begin{align*}
\int\limits_{\frac{\theta \beta^2}{4(\beta-1)}}^{\infty}\int\limits_{0}^{\infty}|g(\lambda,\tau)| d \tau d\lambda
&= \int\limits_{\frac{\theta \beta^2}{4(\beta-1)}}^{\infty}\int\limits_{0}^{\infty}e^{-\lambda \tau}f_t(\tau)|\,b(\lambda)\,\psi(x,-\lambda)\,\psi(x_0,-\lambda)| \, d \tau \, d\lambda \\
&=\int\limits_{\frac{\theta \beta^2}{4(\beta-1)}}^{\infty}\mathcal{E}_{\alpha}(-\lambda t^{\alpha})|b(\lambda)\,\psi(x,-\lambda)\,\psi(x_0,-\lambda)| \, d \lambda \\
&=\int\limits_{\frac{\theta \beta^2}{4(\beta-1)}}^{\infty}|h(\lambda)|\, d \lambda \, ,
\end{align*}
we need to show that
$$\int\limits_{\frac{\theta \beta^2}{4(\beta-1)}}^{\infty}|h(\lambda)| \, d \lambda <\infty.$$
According to \cite{slater1960confluent} or \cite{Buchholz}
$$_{2}F_{0}\left( -\frac{\beta}{2} + ik(\lambda), -\frac{\beta}{2} - ik(\lambda); ; -\frac{x}{\gamma} \right)=\left(\frac{\gamma}{x}\right )^{-\frac{\beta+1}{2}}e^{\frac{\gamma}{2x}}\,\mathcal{W}_{\frac{\beta+1}{2},\, ik(\lambda)}\left(\frac{\gamma}{x}\right),$$
where $\mathcal{W}$ is given by
\begin{equation}\label{Whitt_relation}
\mathcal{W}_{\frac{\beta+1}{2},\, ik(\lambda)}\left(\frac{\gamma}{x}\right)=\frac{\pi}{\sin{(2i\pi k(\lambda))}}\left (\frac{\mathcal{M}_{\frac{\beta+1}{2},\, -ik(\lambda)}\left(\frac{\gamma}{x}\right)}{\Gamma\left( -\frac{\beta}{2}+i k(\lambda) \right)} -\frac{\mathcal{M}_{\frac{\beta+1}{2},\, ik(\lambda)}\left(\frac{\gamma}{x}\right)}{\Gamma\left( -\frac{\beta}{2}-i k(\lambda) \right)} \right )
\end{equation}
and $\mathcal{M}$ is the Whittaker function.

From \cite[p. 94, Equation (1)]{Buchholz}
\begin{equation}\label{Whitt_asymp}\mathcal{M}_{\frac{\beta+1}{2},\, ik(\lambda)}\left(\frac{\gamma}{x}\right) =\frac{\left (\frac{\gamma}{x}\right )^{\frac{1}{2}+ik(\lambda)}}{\Gamma\left(1+2 i k(\lambda)\right)}\left( 1+O ( \left |2 i k(\lambda) \right |^{-1}) \right ), \,\, \lambda \to \infty.
\end{equation}
Using \eqref{Whitt_relation} together with
$$\Gamma\left(2 i k(\lambda)\right )\,\Gamma \left ( 1-2 i k(\lambda) \right )=\frac{\pi}{\sin{(2i\pi k(\lambda))}}$$
we obtain

$$\left | \mathcal{W}_{\frac{\beta+1}{2},\, ik(\lambda)}\left(\frac{\gamma}{x}\right)\right | \leq \left | \Gamma\left(2 i k(\lambda)\right )\,\Gamma \left ( 1-2 i k(\lambda) \right )  \right | \left ( \frac{\left |\mathcal{M}_{\frac{\beta+1}{2},\, -ik(\lambda)}\left(\frac{\gamma}{x}\right)\right |}{\left |\Gamma\left( -\frac{\beta}{2}+i k(\lambda) \right)\right |} +\frac{\left |\mathcal{M}_{\frac{\beta+1}{2},\, ik(\lambda)}\left(\frac{\gamma}{x}\right)\right |}{\left |\Gamma\left( -\frac{\beta}{2}-i k(\lambda) \right)\right |} \right ). $$
Now, \eqref{Whitt_asymp} implies
\begin{align*}\left | \mathcal{W}_{\frac{\beta+1}{2},\, ik(\lambda)}\left(\frac{\gamma}{x}\right)\right | &\leq \left (\frac{\gamma}{x}\right )^{\frac{1}{2}}\left ( \frac{1}{\left |\Gamma \left ( 1+2 i k(\lambda) \right )\right |\,\left |\Gamma\left( -\frac{\beta}{2}+i k(\lambda) \right)\right |} +\frac{1}{\left |\Gamma \left ( 1-2 i k(\lambda) \right )\right |\,\left |\Gamma\left( -\frac{\beta}{2}-i k(\lambda) \right)\right |} \right )  \\
&\times \left | \Gamma\left(2 i k(\lambda)\right )\,\Gamma \left ( 1-2 i k(\lambda) \right )  \right | \,\left( 1+O ( \left |2 i k(\lambda) \right |^{-1}) \right ), \,\, \lambda \to \infty.
\end{align*}
It follows that
\begin{align*}\left | h(\lambda)\right | &\leq  \left (\frac{\gamma^2}{x x_0}\right )^{-\frac{\beta}{2}}e^{\frac{\gamma}{2}\left(\frac{1}{x}+\frac{1}{x_0}\right)} \left ( \frac{1}{\left |\Gamma \left ( 1+2 i k(\lambda) \right )\right |\,\left |\Gamma\left( -\frac{\beta}{2}+i k(\lambda) \right)\right |} +\frac{1}{\left |\Gamma \left ( 1-2 i k(\lambda) \right )\right |\,\left |\Gamma\left( -\frac{\beta}{2}-i k(\lambda) \right)\right |} \right )^2 \\
&\times  \mathcal{E}_{\alpha}(-\lambda t^{\alpha})\, \frac{1}{|k(\lambda)|} \,\left| \frac{ \Gamma^{\frac{1}{2}}(\beta) \, \Gamma\left( -\frac{\beta}{2} + ik(\lambda) \right)}{\Gamma\left(2ik(\lambda)\right)} \right|^{2}
\left | \Gamma\left(2 i k(\lambda)\right )\,\Gamma \left ( 1-2 i k(\lambda) \right )  \right |^2 \,\left( 1+O ( \left |2 i k(\lambda) \right |^{-1}) \right ), \,\, \lambda \to \infty.
\end{align*}
Since
\begin{equation}\label{Gamma_asymp}\Gamma(x+iy) \thicksim \sqrt{2\pi} \cdot |y|^{x-\frac{1}{2}}\cdot e^{-\pi \frac{|y|}{2}}, \,\,\, |y| \to \infty,\end{equation}\\
from \cite[Equation (6.8)]{Simon_Mittag} for $0<\alpha<1$ we have
\begin{equation}\label{Mittag_bound}
\frac{1}{1+\Gamma(1-\alpha)\,\lambda\, t^{\alpha}} \leq \mathcal{E}_{\alpha}(-\lambda t^{\alpha}) \leq \frac{1}{1+\Gamma(1+\alpha)^{-1}\,\lambda \, t^{\alpha}}.
\end{equation}
It follows that
$$\left | h(\lambda)\right | \leq  \left (\frac{\gamma^2}{x x_0}\right )^{-\frac{\beta}{2}}e^{\frac{\gamma}{2}\left(\frac{1}{x}+\frac{1}{x_0}\right)} \frac{1}{1+\Gamma \left ( 1+\alpha \right )^{-1}\lambda t^{\alpha}}\, \frac{\Gamma(\beta)}{|k(\lambda)|}\,\left( 1+O ( \left |2 i k(\lambda) \right |^{-1}) \right ), \,\, \lambda \to \infty$$
and therefore
\begin{equation*}
|h(\lambda)|=O(\lambda^{-\frac{3}{2}}) \text{ as } \lambda \to \infty .
\end{equation*}
Finally, according to \cite[p. 335]{NIST}, since $\lambda \mapsto \psi(x, \lambda)$ is an entire function for a fixed $x$, $\lambda \mapsto |h(\lambda)|$ is also an entire function.
This verifies \eqref{RG_Fubini_proof} and completes the proof of the spectral representation
of the transition density of the fractional RG diffusion.
\end{proof}

\subsection{Spectral representation of the fractional Fisher-Snedecor diffusion} \label{SSS_FS}

The Fisher-Snedecor diffusion satisfies the SDE
\begin{equation*}
dX_{t} = -\theta \left(X_{t} - \frac{\beta}{\beta - 2} \right) \, dt + \sqrt{\frac{4 \theta}{\gamma(\beta - 2)} X_{t} (\gamma X_{t} + \beta) } \, dW_{t}, \quad t \geq 0, \label{FS_sde}
\end{equation*}
with $\theta > 0$ and invariant density
\begin{equation}
\fs(x) =  \frac{ \beta^{\frac{\beta}{2}}}{B\left( \frac{\gamma}{2}, \frac{\beta}{2} \right)} \, \frac{(\gamma x)^{\frac{\gamma}{2} - 1}}{(\gamma x + \beta)^{\frac{\gamma}{2} + \frac{\beta}{2}}} \, \gamma \, {\I}_{\langle 0, \infty \rangle}(x) \label{FSdensity}
\end{equation}
with parameters $\beta > 2$ and $\gamma > 0$, where the latter requirement ensures the existence of the stationary mean $\beta/(\beta-2)$. If $\beta > 4$ the variance of the invariant distribution exists and the autocorrelation function of the FS diffusion is given by \eqref{RGacf}.

The generator is
\begin{equation*}
\G f(x) = \frac{2 \theta}{\gamma(\beta - 2)} \, x^{1 - \frac{\gamma}{2}}(\gamma x + \beta)^{\frac{\gamma}{2} + \frac{\beta}{2}} \, \frac{d}{dx} \, \left(x^{\frac{\gamma}{2}}(\gamma x + \beta)^{1 - \frac{\gamma}{2} - \frac{\beta}{2}} f^{\prime}(x) \right) =
\end{equation*}
\begin{equation*}
\hspace*{5mm} = \frac{2\theta}{\gamma(\beta - 2)} \, x(\gamma x + \beta) f^{\prime \prime}(x) - \theta \left(x - \frac{\beta}{\beta - 2} \right)f^{\prime}(x), \quad x > 0. \label{FSD_InfGen}
\end{equation*}
The structure of the spectrum of $(-\mathcal{G})$ is the same as in the case of the RG diffusion, but with the cutoff $$\Lambda = \frac{\theta \beta^{2}}{4(\beta-1)}, \quad \beta > 2.$$

The discrete part of the spectrum consists of the eigenvalues
\begin{equation}
\lambda_{n} = \frac{\theta}{\beta - 2} \, n (\beta - 2n), \quad n \in \left\{0, 1, \ldots, \left\lfloor \frac{\beta}{4} \right\rfloor \right\}, \quad \beta > 2 \label{FS_diseigen}.
\end{equation}
The corresponding eigenfunctions are FS polynomials given by the Rodrigues formula
\begin{equation*}
\widetilde{F}_{n}(x) = x^{1-\frac{\gamma}{2}} \, (\gamma x + \beta)^{\frac{\gamma}{2} + \frac{\beta}{2}} \, \frac{d^{n}}{dx^{n}} \, \left\{ 2^{n} \, x^{\frac{\gamma}{2} + n - 1} \, (\gamma x + \beta)^{n - \frac{\gamma}{2} - \frac{\beta}{2}} \right\}, \quad n \in \left\{0, 1, \ldots, \left\lfloor \frac{\beta}{4} \right\rfloor \right\}, \quad \beta > 2. \label{FS_RodForm}
\end{equation*}
The normalized FS polynomials are
\begin{equation}
F_{n}(x) = K_{n} \, \widetilde{F}_{n}(x), \label{FS_NormBesselRodForm}
\end{equation}
where
\begin{equation*}
K_{n} = (-1)^{n} \sqrt{\frac{B(\frac{\gamma}{2}, \frac{\beta}{2})}{n! (-1)^{n} (2 \beta)^{2n} B\left( \frac{\gamma}{2} + n, \frac{\beta}{2} - 2n \right)} \, \frac{\Gamma \left( n - \frac{\beta}{2}\right)}{\Gamma \left( 2n - \frac{\beta}{2} \right)}} =
\end{equation*}
\begin{equation*}
\hspace*{1.1cm} = (-1)^{n} \sqrt{\frac{B(\frac{\gamma}{2}, \frac{\beta}{2})}{n! (2 \beta)^{2n} B\left( \frac{\gamma}{2} + n, \frac{\beta }{2} - 2n \right)} \, \left[ \prod_{k=1}^{n} \left( \frac{\beta}{2} + k - 2n \right) \right]^{-1}}.  \label{FS_FSPolNormConst}
\end{equation*}
is the normalizing constant.

Since the absolutely continuous part of the spectrum is of the form $( \Lambda, \infty )$, its elements could be parametrized as
\begin{equation*}
\lambda = \Lambda + \frac{2\theta k^{2}}{\beta-2} = \frac{2\theta}{\beta-2} \left( \frac{\beta^{2}}{16} + k^{2} \right), \quad \beta > 2, \quad k > 0. \label{FS_conteigen}
\end{equation*}
The spectral representation of the transition density of the FS diffusion with parameters $\gamma > 2$ (ensuring the ergodicity), $\gamma \notin \{2(m+1), \, m \in \mathbb{N}\}$, and $\beta > 2$ consists of two parts
\begin{equation}
p_{1}(x; x_{0}, t) = p_{d}(x; x_{0}, t) + p_{c}(x; x_{0}, t).  \label{FS_d+c}
\end{equation}
The discrete part of the spectral representation is
\begin{equation*}
p_{d}(x; x_{0}, t) = {\fs}(x) \, \sum\limits_{n=0}^{\left\lfloor \frac{\beta}{4} \right\rfloor}e^{-\lambda_{n} t} \, F_{n}(x_{0}) \, F_{n}(x),  \label{FS_DisSpec}
\end{equation*}
where $\fs(\cdot)$ is the invariant density \eqref{FSdensity}, eigenvalues $\lambda_{n}$ are given by \eqref{FS_diseigen}
and the normalized FS polynomials are given by \eqref{FS_NormBesselRodForm}. The continuous part of the spectral representation is given in terms of the elements $\lambda$ of the absolutely continuous part of the spectrum of the operator $(-\mathcal{G})$:
\begin{equation*}
p_{c}(x; x_{0}, t) = {\fs}(x) \, \frac{1}{\pi } \int\limits_{\frac{\theta \beta^{2}}{8(\beta-2)}}^{\infty} e^{-\lambda t} \, a(\lambda)
 \, f_{1}(x_{0}, -\lambda) f_{1}(x, -\lambda) \, d\lambda,\label{FS_ContSpec}
\end{equation*}
where
\begin{equation}\label{FSparameters}
k(\lambda) = -i \sqrt{\frac{\beta^{2}}{16} - \frac{\lambda (\beta - 2)}{2 \theta}}, \quad
a(\lambda)=k(\lambda)\left| \frac{ B^{\frac{1}{2}}\left( \frac{\gamma}{2}, \frac{\beta}{2} \right) \, \Gamma\left( -\frac{\beta}{4} + ik(\lambda) \right) \Gamma\left( \frac{\gamma}{2} + \frac{\beta}{4} + ik(\lambda) \right)}{\Gamma\left(\frac{\gamma}{2} \right) \, \Gamma\left(1 + 2ik(\lambda) \right)} \right|^{2}.
\end{equation}
 Function $f_1$ is a solution of the Sturm-Liouville equation $\mathcal{G}f(x) = -\lambda f(x)$, $\lambda > 0$, and is given by
\begin{equation} \label{FS_f1l}
f_{1}(x) = f_{1}(x, s) = \phantom{,}_{2}F_{1} \left(-\frac{\beta}{4} + \sqrt{\frac{\beta^{2}}{16} + \frac{s(\beta-2)}{2\theta}}, -\frac{\beta}{4} - \sqrt{\frac{\beta^{2}}{16}+\frac{s(\beta-2)}{2\theta}}; \frac{\gamma}{2}; -\frac{\gamma}{\beta} \, x \right),
\end{equation}
where $_{2}F_{1}$ is the Gauss hypergeometric function, a special case of generalized hypergeometric function $_{p}F_{q}$ with $p=2$ and $q=1$, see \cite{Slater} or \cite{NIST}.

For more details on FS diffusion and the proof of this representation of the transition density we refer to \cite{AvramLeonenkoSuvak_FS_SR_2013}.

\begin{theorem}
The transition density of fractional FS diffusion is given by
\begin{equation}\label{frtranf}
 p_\alpha (x,t;x_0)=\sum \limits_{n=0}^{\lfloor\frac{\beta}{4}\rfloor} \mathfrak{fs}(x)\, F_n(x_0)\,F_n(x)\,
\mathcal{E}_{\alpha}(-\lambda_n t^{\alpha})
+\frac{\mathfrak{fs}(x)}{\pi } \int\limits_{\frac {\theta \beta^2}{8(\beta-2)}}^{\infty}
\mathcal{E}_{\alpha}(-\lambda t^{\alpha})\,a(\lambda)\,f_1(x_0,-\lambda),f_1(x,-\lambda)\,d\lambda ,
\end{equation}
where the FS polynomials are given by equation \eqref{FS_NormBesselRodForm}, function $f_1$ is given by \eqref{FS_f1l}, and $a(\lambda )$ given by \eqref{FSparameters}.

\end{theorem}

\begin{proof}
Since the FS diffusion $X(t)$ is independent of the time change $E_t$, using \eqref{FS_d+c} and \eqref{Mittag} together with the Fubini argument, we have:
\begin{align} \label{FFS_tonelli}
P(X_{\alpha}(t) \in B | X_{\alpha}(0)=x_0) &= \int_{0}^{\infty}P(X_{1}(\tau) \in B | X_{1}(0)=x_0)\,f_{t}(\tau)\,d\tau \nonumber \\
&=\int_{0}^{\infty}\int_{B}^{}p_1(x,\tau;x_0)\,f_t(\tau)\,dx\, d\tau  \nonumber \\
&=\int_{B}^{}\int_{0}^{\infty}(p_d(x,\tau;x_0)+p_c(x,\tau;x_0))\,f_t(\tau)\,d\tau \,dx
\end{align}
$$=\int_{B}^{} \left[\int\limits_{0}^{\infty} \sum \limits_{n=0}^{\lfloor\frac{\beta}{4}\rfloor} \mathfrak{fs}(x)\, F_n(x_0) \,F_n(x) \,e^{- \lambda_n \tau}\,f_t(\tau)\,d\tau
+\frac{\mathfrak{fs}(x)}{\pi}\int_{0}^{\infty } \int\limits_{\frac{\theta \beta^2}{8(\beta-2)}}^{\infty}e^{-\lambda \tau}\, f_t(\tau)\, a(\lambda)f_1(x_0,-\lambda)\,f_1(x,-\lambda)\,d\lambda \, d\tau  \right ]dx
$$
\begin{equation} \label{FFS_fubini}
=\int_{B}^{}  \left [ \sum \limits_{n=0}^{\lfloor\frac{\beta}{4}\rfloor} \mathfrak{fs}(x)\, F_n(x_0)\,F_n(x)\,\mathcal{E}_{\alpha}(-\lambda_n t^{\alpha})
+\frac{\mathfrak{fs}(x)}{\pi} \int\limits_{\frac{\theta \beta^2}{8(\beta-2)}}^{\infty} \mathcal{E}_{\alpha}(-\lambda t^{\alpha})\,a(\lambda)\,f_1(x_0,-\lambda)\,f_1(x,-\lambda)\,d\lambda  \right ]dx.
\end{equation}
Change of the order of integration in \eqref{FFS_tonelli} is justified by the Fubini-Tonelli Theorem since functions $p_1$ and $f_t$ are non-negative.
Change of the order of integration in \eqref{FFS_fubini} cannot be justified by the Fubini-Tonelli Theorem as in \eqref{FFS_tonelli} since the integrand
$$g(\lambda,\tau)=e^{-\lambda \tau}\,f_t(\tau)\,a(\lambda)\,f_1(x_0,-\lambda)\,f_1(x,-\lambda)$$
is not necessarily non-negative. In order to use the Fubini Theorem, we need to show that
\begin{equation}\label{FS_Fubini_proof}
\int\limits_{\frac{\theta \beta^2}{8(\beta-2)}}^{\infty}\int\limits_{0}^{\infty}|g(\lambda,\tau)|\, d \tau \, d\lambda < \infty.
\end{equation}
Let $$h(\lambda)=\mathcal{E}_{\alpha}(-\lambda t^{\alpha})\,a(\lambda)\,f_1(x_0,-\lambda)\,f_1(x,-\lambda).$$
Since
\begin{align*}
\int\limits_{\frac{\theta \beta^2}{8(\beta-2)}}^{\infty}\int\limits_{0}^{\infty}|g(\lambda,\tau)|\, d \tau \, d\lambda
&= \int\limits_{\frac{\theta \beta^2}{8(\beta-2)}}^{\infty}\int\limits_{0}^{\infty}e^{-\lambda \tau}\,f_t(\tau)\,|a(\lambda)\,f_1(x_0,-\lambda)\,f_1(x,-\lambda)|\, d \tau \,d\lambda \\
&=\int\limits_{\frac{\theta \beta^2}{8(\beta-2)}}^{\infty}\mathcal{E}_{\alpha}(-\lambda t^{\alpha})\,|a(\lambda)\,f_1(x_0,-\lambda)\,f_1(x,-\lambda)|\, d \lambda \\
&=\int\limits_{\frac{\theta \beta^2}{8(\beta-2)}}^{\infty}|h(\lambda)|\, d \lambda \,
\end{align*}
we need to show that
$$\int\limits_{\frac{\theta \beta^2}{8(\beta-2)}}^{\infty}|h(\lambda)|\, d \lambda <\infty.$$
From \cite[p. 77, Equation (17)]{Erdelyi},
\begin{align*}f_1(x,-\lambda) &=\frac{\Gamma\left(1+\frac{\beta}{4}+ik(\lambda)\right)\,\Gamma\left(\frac{\gamma}{2}\right)}{\Gamma\left(\frac{1}{2}\right)\,\Gamma\left(\frac{\gamma}{2}+\frac{\beta}{4}+ik(\lambda)\right )}\cdot 2^{-\frac{\beta}{2}-1}\cdot\left(1-e^{\xi}\right)^{-\frac{\gamma}{2}+\frac{1}{2}}\cdot\left(1+e^{\xi}\right)^{\frac{\gamma}{2}+\frac{\beta}{2}-\frac{1}{2}}\\ &\cdot \left (ik(\lambda)\right )^{-\frac{1}{2}} \cdot \left ( e^{\xi(\frac{\beta}{4}+ik(\lambda))}+e^{i\pi(\frac{\gamma}{2}-\frac{1}{2})}\cdot e^{\xi(\frac{\beta}{4}-ik(\lambda))} \right )\left ( 1+O(|k(\lambda)|^{-1})\right ), \, \, \lambda \to \infty,
\end{align*}
where $e^{\xi}=1+\frac{2\gamma}{\beta}x+\sqrt{\frac{4\gamma}{\beta}x(1+\frac{\gamma}{\beta}x)}.$\\
It follows that
\begin{equation*}
\begin{split}
|f_1(x,-\lambda)| &\leq \left|\frac{\Gamma\left(1+\frac{\beta}{4}+ik(\lambda)\right)\,\Gamma\left(\frac{\gamma}{2}\right)}{\Gamma\left(\frac{1}{2}\right)\,\Gamma\left(\frac{\gamma}{2}+\frac{\beta}{4}+ik(\lambda)\right)}\right|2^{-\frac{\beta}{2}-1}\cdot|1-e^{\xi}|^{-\frac{\gamma}{2}+\frac{1}{2}}\cdot |1+e^{\xi}|^{\frac{\gamma}{2}+\frac{\beta}{2}-\frac{1}{2}}\\ &\times \left |k(\lambda) \right |^{-\frac{1}{2}}\cdot e^{\xi\frac{\beta}{4}} \left ( 1+|e^{i\pi(\frac{\gamma}{2}-\frac{1}{2})}| \right )\left ( 1+O(|k(\lambda)|^{-1})\right ), \, \, \lambda \to \infty.
\end{split}
\end{equation*}
Now we have
\begin{align*}|h(\lambda)| &\leq \mathcal{E}_{\alpha}(-\lambda t^{\alpha})\left| \frac{B^{\frac{1}{2}}\left( \frac{\gamma}{2}, \frac{\beta}{2} \right) }{\Gamma\left (\frac{1}{2}\right )} \right|^{2}  \left|\frac{\Gamma\left( -\frac{\beta}{4} + ik(\lambda) \right) \, \Gamma\left (1+\frac{\beta}{4}+ik(\lambda)\right )}{\Gamma\left( 1 + 2ik(\lambda) \right) }\right|^2 2^{-\beta-2}\cdot|(1-e^{\xi})(1-e^{\xi_0})|^{-\frac{\gamma}{2}+\frac{1}{2}}\\
&\times \left|(1+e^{\xi})(1+e^{\xi})\right|^{\frac{\gamma}{2}+\frac{\beta}{2}-\frac{1}{2}}  \cdot e^{(\xi+\xi_0)\frac{\beta}{4}} \left ( 1+|e^{i\pi(\frac{\gamma}{2}-\frac{1}{2})}| \right )^2 \left ( 1+O(|k(\lambda)|^{-1})\right ), \, \, \lambda \to \infty.
\end{align*}
Using \eqref{Gamma_asymp} and \eqref{Mittag_bound} we obtain
\begin{align*}|h(\lambda)| &\leq \frac{1}{1+\Gamma(1+\alpha)^{-1}\,\lambda \, t^{\alpha}}\cdot B\left(\frac{\gamma}{2},\frac{\beta}{2}\right)\cdot \left|k(\lambda)\right|^{-1}   2^{-\beta-2}\cdot|(1-e^{\xi})(1-e^{\xi_0})|^{-\frac{\gamma}{2}+\frac{1}{2}}\\
&\times \left|(1+e^{\xi})(1+e^{\xi})\right|^{\frac{\gamma}{2}+\frac{\beta}{2}-\frac{1}{2}}  \cdot e^{(\xi+\xi_0)\frac{\beta}{4}} \left ( 1+|e^{i\pi(\frac{\gamma}{2}-\frac{1}{2})}| \right )^2 \left ( 1+O(|k(\lambda)|^{-1})\right ), \, \, \lambda \to \infty.
\end{align*}
It follows that $$|h(\lambda)|=O(\lambda^{-\frac{3}{2}}), \text{ as } \lambda \to \infty.$$
Finally, according to \cite[p. 68]{Erdelyi}, $\lambda \mapsto f_1(x, \lambda)$ is an entire function for a fixed $x$ and so $\lambda \mapsto |h(\lambda)|$ is also an entire function. This verifies \eqref{FS_Fubini_proof} and completes the proof of the spectral representation of the transition density of the fractional FS diffusion.
\end{proof}

\section{Stationary distributions of the fractional reciprocal gamma and Fisher-Snedecor diffusions}

We now show that as $t\to \infty $, the distribution of $X_\alpha (t)$ approaches the stationary distribution of the non-fractional FS or RG Pearson diffusion, respectively.
\begin{theorem}
Let $X_\alpha (t)$ be reciprocal gamma or Fisher-Snedecor fractional diffusion defined by \eqref{fpddef} and let $p_\alpha (x,t)$ be the density of $X_\alpha (t)$. Assume that $X_\alpha (0)$ has a twice continuously differentiable density $f$ that vanishes at infinity. Then
\begin{equation*}
p_{\alpha}(x,t)\to m (x)\,\, as \,\, t \to \infty,
\end{equation*}
where $m (\cdot )$ is the density of the non-fractional stationary reciprocal gamma or Fisher-Snedecor diffusion, respectively.
\end{theorem}

\begin{proof}
Using the definition of the transition density $p_\alpha (x,t;y)$  \eqref{trans}, we have
\[
p_{\alpha}(x,t)=\int _0^\infty p_\alpha (x,t;y) f(y)dy
\]
and therefore it suffices to prove that
\begin{equation*}
p_{\alpha}(x,t;y)\to m (x) \quad \mathrm{as} \quad t \to \infty
\end{equation*}
for fixed $x$ and $y$. This together with the fact that $f(y)$ and $p_{\alpha}(x,t;y)$ are density functions then yields
\[
\int _0^\infty p_\alpha (x,t;y) f(y)dy \to m (x) \int _0^\infty  f(y)dy=m (x) \,\, \mathrm{ as } \,\, t \to \infty.
\]
We treat the reciprocal gamma and Fisher-Snedecor cases separately.

\bigskip \bigskip

\textbf{Reciprocal gamma diffusion}

\bigskip

Since $\lambda_0=0$, it follows that
\begin{align*}
p_{\alpha}(x,t;y)&= \sum \limits_{n=0}^{\lfloor\frac{\beta}{2}\rfloor} \mathfrak{rg}(x)\, B_n(x)\,\,B_n(y)\,\mathcal{E}_{\alpha}(-\lambda_n t^{\alpha})
+\frac{\mathfrak{rg}(x)}{4\pi} \int\limits_{\frac{\theta \beta^2}{4(\beta-1)}}^{\infty} \mathcal{E}_{\alpha}(-\lambda t^{\alpha})\,b(\lambda)\,\psi(x,-\lambda)\,\psi(y,-\lambda)\,d\lambda \\
&= \mathfrak{rg}(x)+\sum \limits_{n=1}^{\lfloor\frac{\beta}{2}\rfloor} \mathfrak{rg}(x) \,B_n(x)\,B_n(y)\,\mathcal{E}_{\alpha}(-\lambda_n t^{\alpha})
+\frac{\mathfrak{rg}(x)}{4\pi} \int\limits_{\frac{\theta \beta^2}{4(\beta-1)}}^{\infty} \mathcal{E}_{\alpha}(-\lambda t^{\alpha})\,b(\lambda)\,\psi(x,-\lambda)\,\psi(y,-\lambda)\,d\lambda. \\
\end{align*}
For a constant $c$ such that
$$c \geq \frac{4(\beta-1)}{\theta \beta^2 t^{\alpha}}+\Gamma(1-\alpha)\geq \frac{1}{\lambda t^{\alpha}}+\Gamma(1-\alpha)>0$$
from \eqref{Mittag_bound} we obtain
$$\frac{1}{c \lambda t^{\alpha}}\leq \frac{1}{1+\Gamma(1-\alpha)\lambda t^{\alpha}}\leq\mathcal{E}_{\alpha}(-\lambda t^{\alpha}) \leq\frac{1}{1+\Gamma(1+\alpha)^{-1}\lambda t^{\alpha}}.$$
Now it follows that
$$\frac{1}{t^{\alpha}}\int\limits_{\frac{\theta \beta^2}{4(\beta-1)}}^{\infty} \frac{1}{c \lambda}\,b(\lambda)\,\psi(x,-\lambda)\,\psi(y,-\lambda)\,d\lambda \leq\int\limits_{\frac{\theta \beta^2}{4(\beta-1)}}^{\infty} \mathcal{E}_{\alpha}(-\lambda t^{\alpha})\,b(\lambda)\,\psi(x,-\lambda)\,\psi(y,-\lambda)\,d\lambda,$$
$$\int\limits_{\frac{\theta \beta^2}{4(\beta-1)}}^{\infty} \mathcal{E}_{\alpha}(-\lambda t^{\alpha})\,b(\lambda)\,\psi(x,-\lambda)\,\psi(y,-\lambda)\,d\lambda \leq \frac{1}{t^{\alpha}}\int\limits_{\frac{\theta \beta^2}{4(\beta-1)}}^{\infty} \frac{\Gamma(1+\alpha)}{\lambda}\,b(\lambda)\,\psi(x,-\lambda)\,\psi(y,-\lambda)\,d\lambda .$$
Letting $t \to \infty$ yields
$$\int\limits_{\frac{\theta \beta^2}{4(\beta-1)}}^{\infty} \mathcal{E}_{\alpha}(-\lambda t^{\alpha})\,b(\lambda)\,\psi(x,-\lambda)\,\psi(y,-\lambda)\,d\lambda \to 0\,, \,\, t \to \infty$$
and
$$\sum \limits_{n=1}^{\lfloor\frac{\beta}{2}\rfloor} \mathfrak{rg}(x) \,B_n(x)\,B_n(y)\,\mathcal{E}_{\alpha}(-\lambda_n t^{\alpha}) \to 0, \,\, t \to \infty .$$
Therefore
\[
p_{\alpha}(x,t;y) \to \mathfrak{rg}(x), \,\, \,\, x>0 \text{ as } t \to \infty.
\]

\bigskip \bigskip

\textbf{Fisher-Snedecor diffusion}

\bigskip

Since $\lambda_0=0$, it follows that
\begin{align*}
p_{\alpha}(x,t;y)&= \sum \limits_{n=0}^{\lfloor\frac{\beta}{4}\rfloor} \mathfrak{fs}(x)\, F_n(y)\,F_n(x)\,\mathcal{E}_{\alpha}(-\lambda_n t^{\alpha})
+\frac{\mathfrak{fs}(x)}{\pi } \int\limits_{\frac{\theta \beta^2}{8(\beta-2)}}^{\infty} \mathcal{E}_{\alpha}(-\lambda t^{\alpha})\,a(\lambda)\,f_1(y,-\lambda)\,f_1(x,-\lambda)\,d\lambda \\
&= \mathfrak{fs}(x)+\sum \limits_{n=1}^{\lfloor\frac{\beta}{4}\rfloor} \mathfrak{fs}(x)\, F_n(y)\,F_n(x)\,\mathcal{E}_{\alpha}(-\lambda_n t^{\alpha})
+\frac{\mathfrak{fs}(x)}{\pi } \int\limits_{\frac{\theta \beta^2}{8(\beta-2)}}^{\infty} \mathcal{E}_{\alpha}(-\lambda t^{\alpha})\,a(\lambda)\,f_1(y,-\lambda)\,f_1(x,-\lambda)\,d\lambda.
\end{align*}
Let $c$ be a constant such that
$$c \geq \frac{8(\beta-2)}{\theta \beta^2 t^{\alpha}}+\Gamma(1-\alpha)\geq \frac{1}{\lambda t^{\alpha}}+\Gamma(1-\alpha)>0.$$
From \eqref{Mittag_bound} we obtain
$$\frac{1}{c \lambda t^{\alpha}}\leq \frac{1}{1+\Gamma(1-\alpha)\lambda t^{\alpha}}\leq\mathcal{E}_{\alpha}(-\lambda t^{\alpha}) \leq\frac{1}{1+\Gamma(1+\alpha)^{-1}\lambda t^{\alpha}}.$$
Therefore,
$$\frac{1}{ t^{\alpha}}\int\limits_{\frac{\theta \beta^2}{8(\beta-2)}}^{\infty} \frac{1}{c\lambda  }\,a(\lambda)\,f_1(y,-\lambda)\,f_1(x,-\lambda)\,d\lambda \leq \int\limits_{\frac{\theta \beta^2}{8(\beta-2)}}^{\infty} \mathcal{E}_{\alpha}(-\lambda t^{\alpha})\,a(\lambda)\,f_1(y,-\lambda)\,f_1(x,-\lambda)\,d\lambda,$$
$$\int\limits_{\frac{\theta \beta^2}{8(\beta-2)}}^{\infty} \mathcal{E}_{\alpha}(-\lambda t^{\alpha})\,a(\lambda)\,f_1(y,-\lambda)\,f_1(x,-\lambda)\,d\lambda \leq\frac{1}{ t^{\alpha}}\int\limits_{\frac{\theta \beta^2}{8(\beta-2)}}^{\infty} \frac{\Gamma(1+\alpha)}{\lambda  }\,a(\lambda)\,f_1(y,-\lambda)\,f_1(x,-\lambda)\,d\lambda .$$
Letting $t \to \infty$ yields
$$\int\limits_{\frac{\theta \beta^2}{8(\beta-2)}}^{\infty} \mathcal{E}_{\alpha}(-\lambda t^{\alpha})\,a(\lambda)\,f_1(y,-\lambda)\,f_1(x,-\lambda)\,d\lambda \,\, \to \,\, 0, \text{ as } t \to \infty$$
and
$$\sum \limits_{n=1}^{\lfloor\frac{\beta}{4}\rfloor} \mathfrak{fs}(x)\, F_n(y)\,F_n(x)\,\mathcal{E}_{\alpha}(-\lambda_n t^{\alpha}) \to 0, \text{ as } t \to \infty .$$
Therefore
\[
p_{\alpha}(x,t;y) \to \mathfrak{fs}(x), \,\, \,\, x>0 \text{ as } t \to \infty.
\]
\end{proof}

\section{Correlation structure of fractional Pearson diffusions}

Assume that $X(t)$ is a stationary Pearson diffusion and that its parameters are such that the stationary distribution has finite second moment. Then the correlation function  of $X(t)$ is given by
\begin{equation}\label{korelacijaPD}
\Corr{\left[ X(t),\,X(s) \right] } = \exp(-\theta|t-s|),
\end{equation}
where $\theta$ is the autocorrelation parameter. Since the autocorrelation function \eqref{korelacijaPD} falls off exponentially, Pearson diffusions exhibit short-range dependence.

We say that fractional Pearson diffusion $X_{\alpha} (t)$ defined by \eqref{fpddef} is in the steady state if it starts from its invariant distribution with the density $m$. Then the autocorrelation function of $X_{\alpha}(t) = X(E_t)$ is given by
\begin{equation} \label{korelacija_FPD}
\Corr{\left[ X_{\alpha}(t),\,X_{\alpha}(s)\right]}=\mathcal{E}_{\alpha}(-\theta t^{\alpha})+\frac{\theta \alpha t^{\alpha}}{\Gamma(1+\alpha)}\int\limits_{0}^{s/t}\frac{\mathcal{E}_{\alpha}(-\theta t^{\alpha} (1-z)^{\alpha})}{z^{1-\alpha}}dz
\end{equation}
for $t \ge s >0.$

The proof of this fact for non-heavy-tailed fractional Pearson diffusions is given in \cite{LeonenkoMeerschaertSikorskii_CSFPD_2013}. The proof does not depend on the type of invariant Pearson distribution, and therefore the same proof can be repeated for all three heavy-tailed fractional Pearson diffusions, provided that the tails are not too heavy so that the second moment of the corresponding heavy-tailed Pearson distribution exists.

The autocorrelation function \eqref{korelacija_FPD} falls off like power law with exponent $\alpha \in (0,1)$, i.e. for any fixed $s>0$
\[
\Corr [X_{\alpha}(t),X_{\alpha}(s)] = \frac{1}{t^{\alpha}\Gamma(1-\alpha)}\left ( \frac{1}{\theta}+\frac{s^{\alpha}}{\Gamma{(1+\alpha)}} \right )(1+o(1)) \text{ as } t \to \infty.
 \]
Therefore, unlike non-fractional Pearson diffusions, their fractional analogues are long-range dependent processes in the sense that their correlation functions decay slowly.

\section{Strong solutions of time-fractional backward Kolmogorov equation}

To establish the main result of this section, we need the following Lemma.

\begin{lemma}\label{continuity}
For the reciprocal gamma and Fisher-Snedecor diffusions, the family of operators
\begin{equation*}
T_{t} g(y)=E[g(X(t))\,|\,X(0)=y],\quad t\ge 0
\end{equation*}
forms a strongly continuous bounded ($C_0$) semigroup on the space of bounded continuous functions $g$ on $[0, \infty)$ vanishing at infinity.

\end{lemma}

\begin{proof}
The proof of this Lemma is the same as for the non-heavy-tailed diffusions considered in \cite{LeonenkoMeerschaertSikorskii_FPD_2013}. We provide it here for completeness.
The semigroup property follows from the Chapman-Kolmogorov equation for the reciprocal gamma and Fisher-Snedecor diffusions, and uniform boundedness of the semigroup on the above Banach space of continuous functions with the supremum norm  follows from \cite[Theorem 3.4]{F_1975}. Therefore, the family of operators $\{T(t),\,t\ge 0\}$ forms a uniformly bounded semigroup on the respective Banach space of continuous functions, with the supremum norm. Next, we show the pointwise continuity of the semigroup.  For any fixed $y\in (l,L)$
\begin{equation*}\begin{split}
T(t)g(y)-g(y)&=\int _l^L p_1(x,t;y)(g(x)-g(y))dx\\
&=\int _{|x-y| \le \epsilon \cap (l,L)}p_1(x,t;y)(g(x)-g(y))dx\\
&\quad\quad\quad + \int _{|x-y|> \epsilon \cap (l,L)}p_1(x,t;y)(g(x)-g(y))dx\\
&\le \sup _{|x-y| \le \epsilon \cap (l,L)}|g(x)-g(y)|\int _{|x-y| \le \epsilon \cap (l,L)}p_1(x,t;y)dx\\
&\quad\quad\quad +C\int _{|x-y|> \epsilon \cap (l,L)}p_1(x,t;y)dx
\end{split}\end{equation*}
since the function $g$ is bounded. Since $\int _{|x-y|> \epsilon \cap (l,L)}p_1(x,t;y)dx\to 0$ as $t\to 0$ for any $\epsilon >0$ (see  \cite{KarlinTaylor2_1981}, p. 158), the second term in the above expression tends to zero as $t\to 0$. The first term is bounded by $\sup _{|x-y| \le \epsilon \cap (l,L)}|g(x)-g(y)|$, which tends to zero as $\epsilon \to 0$ because of the continuity of $g$. Pointwise continuity then implies strong continuity in view of  \cite[Lemma 6.7]{RW_1994}.

\end{proof}

The next result gives a strong solution to the fractional Cauchy problem associated with the time-fractional backward Kolmogorov equation.

\begin{theorem}
For any $g$ from the domain of the generator $\mathcal G$ specified in \eqref{domain}, a strong solution to the fractional Cauchy problem \eqref{FbackwardCauchy} is given by
\begin{equation}\label{fbackwardsolution}
q(t;y)=\int _l^L p_\alpha (x,t;y)g(x)dx,
\end{equation}
where the transition density $p_\alpha $ is given by equation \eqref{frtranrg} in the reciprocal gamma case and by equation \eqref{frtranf} in the Fisher-Snedecor case.
\end{theorem}

\begin{proof}
The proof of this Theorem consists of several steps. First, Lemma \ref{continuity} and \cite[Proposition 3.1.9]{Arendt} show that $q(t; y)=T(t)g(y)$ solves the non-fractional Cauchy problem \eqref{backwardCauchy}. Second, strong continuity of the semigroup in the Banach space of continuous functions with the supremum norm and Theorem 3.1 in \cite{BaeumerMeerschaert_2001} show that
\begin{equation}\label{SGFCPSOL}
S_t g (y) = \int\limits_{0}^{\infty}T_u g(y) \,f_t(u)\,du,
\end{equation}
where $f_t$ is the density of the inverse stable subordinator $E_t$ given by \eqref{etdensity}, solves the fractional Cauchy problem \eqref{backwardCauchy} for any $g$ from the domain of the generator $\mathcal G$.

Third, since
\begin{eqnarray*}
S_t g(y) &=& \int\limits_{0}^{\infty}T_{u}\, g(y)\, f_t(u)\,du \\
         &=& \int\limits_{0}^{\infty}E[g(X(t))\,|\,X(0)=y]\, f_t(u)\,du\\
         &=& E[g(X(E_t))\,|\,X(E_0)=y]\\
         &=& E[g(X_{\alpha}(t))\,|\,X_{\alpha}(0)=y]
\end{eqnarray*}
where $E_0=0$ almost surely and
\[
E[g(X_{\alpha}(t))\,|\,X_{\alpha}(0)=y]=\int p_{\alpha}(x,t;y)\,g(x)\,dx,
\]
a strong solution to \eqref{FbackwardCauchy} is given by \eqref{fbackwardsolution}.

\end{proof}

\begin{remark}
The explicit expressions for strong solutions of the fractional Cauchy problem associated with the generators of the reciprocal gamma and Fisher-Snedecor diffusions are:
\begin{equation*} \label{FPD_tranrg}
\begin{split}
u_{\mathfrak{rg}}(t;y)&=  \sum \limits_{n=0}^{\lfloor\frac{\beta}{2}\rfloor} \,B_n(y)\,\mathcal{E}_{\alpha}(-\lambda_n t^{\alpha})\int\limits_{0}^{\infty} B_n(x)\,\mathfrak{rg}(x)\,g(x)\,dx \\
& +  \int\limits_{0}^{\infty}\frac{\mathfrak{rg}(x)\,g(x)}{4\pi} \int\limits_{\frac{\theta \beta^2}{4(\beta-1)}}^{\infty} \mathcal{E}_{\alpha}(-\lambda t^{\alpha})\,b(\lambda)\,\psi(x,-\lambda)\,\psi(y,-\lambda)\,d\lambda\, dx
\end{split}
\end{equation*}
and
\begin{equation*} \label{FPD_tranf}
\begin{split}
u_{\mathfrak{fs}}(t;y)&=  \sum \limits_{n=0}^{\lfloor\frac{\beta}{4}\rfloor}  F_n(y)\mathcal{E}_{\alpha}(-\lambda_n t^{\alpha})\int\limits_{0}^{\infty}F_n(x)\,\mathfrak{fs}(x)\,g(x)\,dx\\
& + \frac{1}{\pi} \int\limits_{0}^{\infty}\mathfrak{fs}(x)\,g(x)\int\limits_{\frac{\theta \beta^2}{8(\beta-2)}}^{\infty} \mathcal{E}_{\alpha}(-\lambda t^{\alpha})a(\lambda)f_1(y,-\lambda)f_1(x,-\lambda)d\lambda\,dx.
\end{split}
\end{equation*}

\end{remark}

The explicit expressions for strong solutions of the Cauchy problem for the fractional Fokker-Planck equations were obtained in \cite{LeonenkoMeerschaertSikorskii_FPD_2013} for all three non-heavy-tailed fractional Pearson diffusions using their spectral properties. Since the structure of the spectrum for the reciprocal gamma and Fisher-Snedecor diffusions is much more complex than in the non-heavy-tailed cases, strong solutions of Cauchy problems associated with the fractional Fokker-Planck equation are not presented here. Below we state the result on the $L^2$ solutions. Proving that these are also strong solutions that hold pointwise remains an open problem at this time.

\begin{theorem}
The fractional Cauchy problem
\begin{equation}\label{forward}
\frac {\partial ^\alpha q(x,t)}{\partial t^\alpha }=-\frac{\partial}{\partial x}\left(\mu(x)q(x,t) \right)+\frac{1}{2}\frac{\partial^2}{\partial x^2}\left(
        \sigma^2(x) q(x,t) \right),\,\,
				q(x,0)=f(x),
 \end{equation}
where $f$ is twice continuously differentiable function that vanishes at zero and has a compact support, is solved by
\begin{equation}\label{sol}
q(x,t)=\int _0^\infty p_\alpha (x,t;y)f(y)dy.
\end{equation}
The transition density $p_\alpha $ is given by equation \eqref{frtranrg} in the reciprocal gamma case and by equation \eqref{frtranf} in the Fisher-Snedecor case, and the solution is
 in the following sense: for every $t>0$, $q(x,t)$ given by \eqref{sol} satisfies \eqref{forward}, and the equality holds
in the space of functions
$\{ q (\cdot , t) \in L^{2}\left((0, \infty )\right)\}$.
\end{theorem}

\begin{proof}
As discussed in Section \ref{PD}, the generator of the reciprocal gamma and Fisher-Snedecor diffusions defined on the space of functions \eqref{domain}
with $(l, L)=(0, \infty )$ is self-adjoint. We consider the space $L^2((0,\infty ))$ without the weight $\m $, and the generator defined on a subset of its domain, namely on the set of  functions $f \in L^{2}\left((0,\infty )\right) \cap C^2\left( (0,\infty ) \right)$ such that $ f $ vanishes at 0 and has compact support.

The Fokker-Planck operator
\begin{equation*}
Lf(x)=-\frac{\partial}{\partial x}\left(\mu(x)f(x) \right)+\frac{1}{2}\frac{\partial^2}{\partial x^2}\left(
        \sigma^2(x) f(x) \right)
\end{equation*}
is adjoint to the generator of the diffusion $\mathcal G$ on this subset of $L^2((0,\infty ))$. The semigroup $T_t$ is a $C_0$-semigroup in this space as well as in the space of continuous functions with the supremum norm. From \cite[Corollary 10.6]{Pazy_1983}, the adjoint semigroup $T^*_t$
\begin{equation*}
T^*_tf(x)=\int _l^L p (x,t;y)f(y)dy
\end{equation*}
is the $C_0$-semigroup as well, and its generator is the Fokker-Planck operator.

Since $T^*_t f$ solves the non-fractional Cauchy problem for the Fokker-Planck equation
 \cite[Proposition 3.1.9] {Arendt}, the application of Theorem 3.1 in \cite{BaeumerMeerschaert_2001} completes the proof.
\end{proof}

\bigskip \bigskip

\textbf{Acknowledgements} \newline

The authors would like to thank to Prof. Kre\v{s}imir Burazin
(Department of Mathematics, J.J. Strossmayer University of Osijek) and to
Prof. \'{A}rp\'{a}d Baricz (Department of Economics, Babe\c o-Bolyai
University, Cluj-Napoca, Romania and Institute of Applied Mathematics, John
von Neumann Faculty of Informatics, \'{O}buda University, Budapest,
Hungary) for many useful remarks and discussions.

 N. Leonenko was supported in particular by Cardiff Incoming Visiting Fellowship
Scheme and International Collaboration Seedcorn Fund,
  Cardiff  Data Innovation Research Institute Seed Corn Funding, Australian Research
Council's Discovery Projects funding scheme (project number DP160101366), and by
projects MTM2012-32674 and MTM2015--71839--P (co-funded with Federal funds), of the
DGI, MINECO, Spain.

I. Papi\'c and N. \v{S}uvak were supported by the project number IZIP-2014-7
"Fractional Pearson diffusions" funded by
the J.~J. Strossmayer University of Osijek.

A. Sikorskii was supported in part by grants from the National Institutes of Health R01CA193706, R01CA162401, R01CA157459, R01HD073296, R01HD070723.

\newpage

\bibliographystyle{agsm}

\bibliography{FPD_references}

\end{document}